\documentclass[10pt,twoside]{amsart}
\usepackage{amssymb, amsmath, mathrsfs, mathtools, amsfonts, amsthm, amscd, yfonts}
\usepackage{xcolor}
\usepackage{hyperref} 
\usepackage{multicol}

\theoremstyle{plain}
\newtheorem{Theorem}{Theorem}
\newtheorem{Proposition}[Theorem]{Proposition}
\newtheorem{Lemma}[Theorem]{Lemma}
\newtheorem{Corollary}[Theorem]{Corollary}
\newtheorem{Conjecture}[Theorem]{Conjecture}
\theoremstyle{definition}
\newtheorem{Definition}[Theorem]{Definition}
\theoremstyle{remark}
\newtheorem{Remark}[Theorem]{Remark}
\newtheorem{Example}[Theorem]{Example}
\title{Primeness property for central polynomials of regular gradings}

\title{Primeness property for regular gradings}

\author[L. Centrone]{Lucio Centrone}
\address{Dipartimento di Matematica, Universit\`a degli Studi di Bari Aldo Moro, Via Edoardo Orabona, 4, 70125 Bari, Italy}
\email{lucio.centrone@uniba.it}
\author[C. Fideles]{Claudemir Fideles}
\thanks{C. Fidelis was partially supported by CNPq grant No.~305651/2021-8 and FAPESP grant No.~2023/04011-9}
\address{Department of Mathematics, UNICAMP, 13083-859 Campinas, SP,  Brazil}
\email{fideles@unicamp.br}
\thanks{L. Centrone was partially supported by PNRR-MUR PE0000023-NQSTI}
\author[P. Koshlukov]{Plamen Koshlukov}
\address{IMECC, UNICAMP, Rua S\'ergio Buarque de Holanda 651, 13083-859 Campinas, SP, Brazil}
\email{plamen@unicamp.br}
\thanks{P. Koshlukov was partially supported by FAPESP Grant 2024/14914-9, and CNPq Grant 307184/2023-4}
\author[K. Pereira]{Kau\^e Pereira}
\address{IMECC, UNICAMP, Rua S\'ergio Buarque de Holanda 651, 13083-859 Campinas, SP, Brazil}
\email{k200608@dac.unicamp.br}
\thanks{K. Pereira was supported by FAPESP Grant 2023/01673-0, and FAPESP Grant 2025/03763-2}



\begin{document}
\begin{abstract}
Let $K$ be an algebraically closed field of characteristic 0 and let $G$ be a finite abelian group. If $A$ is a $G$-graded $K$-algebra, the notion of a graded polynomial identity for $A$ is defined in a natural way. A fundamental notion in PI theory is that of a central polynomial: a polynomial in the free associative algebra that, when evaluated on $A$, assumes only central values. A graded central polynomial for $A$ is defined in the same way. Suppose $A=\oplus_{g\in G} A_g$ is the decomposition of $A$ as a direct sum of its homogeneous (in the $G$-grading) components. Suppose that for every $n$ and for every $n$-tuple $(g_1,\ldots, g_n)$ of elements of $G$ there exist $a_i\in A_{g_i}$ such that $a_1\cdots a_n\ne 0$. Suppose further that for every $g$, $h\in G$ there exists a scalar $\beta(g,h)\in K^*$ such that for every $a_g\in A_g$ and $a_h\in A_h$ it holds $a_ga_h=\beta(g,h) a_ha_g$. Then the grading is called regular. A regular grading is minimal whenever there exist no $g$, $h\in G$, $g\ne h$, such that $\beta(g,x)=\beta(h,x)$ for every $x\in G$. 

In 1979, A. Regev proved that if $f$ and $g$ are two polynomials in disjoint sets of variables such that $fg$ is a central polynomial for the matrix algebra $M_n(K)$ then both $f$ and $g$ are central polynomials. Thus we have a sort of primeness property for the central polynomials. We study a generalization of the notion of primeness for graded central polynomials. First we prove that if $A$ is $G$-graded regular algebra then it does not satisfy the primeness property for graded central polynomials. As a consequence, the matrix algebras $M_n(K)$, equipped with the so-called Pauli grading, do not satisfy the primeness property. All these facts allow us to obtain several positive results, concerning graded matrix algebras. It turns out that the primeness property of the graded central polynomials is related to elementary gradings on the matrix algebra. Recall a grading is elementary whenever all matrix units are homogeneous in the grading. This leads us to study the matrix algebras of small order. We prove that, for matrices of orders $2$ and $3$, there exist no nontrivial gradings satisfying the primeness property for central polynomials, and we conjecture that this result holds for $n>3$. Finally we study the primeness property for ordinary central polynomials for $\mathbb{Z}_{2}$-graded regular algebras. We use the known fact that if a $\mathbb{Z}_{2}$-grading on $A$ is regular and minimal then $A$ satisfies the graded identities of the infinite dimensional Grassmann algebra $E$, and moreover, $A$ contains a copy of $E$. We obtain that such algebras $A$ satisfy the primeness property for central polynomials. As a consequence we prove that one can drop the condition on minimality for the regularity of the grading. 
\end{abstract}
\maketitle

\section{Introduction}
Let $A$ be an associative algebra over a field $K$ of characteristic 0, and let $K\langle X\rangle$ be the free associative algebra freely generated by the infinite set $X$ over $K$. Denote by $T(A)\subseteq K\langle X\rangle$ the set of all polynomial identities for $A$. Then $T(A)$ is a T-ideal, that is, it is an ideal that is closed under all endomorphisms of $K\langle X\rangle$. The quotient $K\langle X\rangle/T(A)$ is the relatively free algebra of $A$, it satisfies the same identities as $A$. But the description of $T(A)$, or equivalently, of the relatively free algebra of $A$, is extremely difficult and has been achieved in very few cases. Around 1985, Kemer developed a structure theory of the T-ideals, see for examle \cite{kemer1985varieties}. He described the T-prime T-ideals (the ones that are "prime" with respect to the class of the T-ideals), proved that in a sense, the T-prime T-ideals are the "building blocks" of all T-ideals, and moreover solved in the affirmative the long standing Specht problem. The latter problem asked whether every T-ideal in characteristic 0 is finitely generated. The methods developed by Kemer rely on the systematic use of gradings by the group $\mathbb{Z}_2$, and on the interplay between the gradings and grade identities. Let $E$ be the infinite dimensional Grassmann algebra of a vector space $V$ with a basis $e_1$, $e_2$, \dots{} Then $E$ has a basis consisting of 1 and all monomials $e_{i_1}e_{i_2}\cdots e_{i_k}$, $i_1<i_2<\cdots <i_k$, $k\ge 1$. The multiplication in $E$ is induced by $e_ie_j=-e_je_i$ for every $i$ and $j$. The Grassmann algebra admits a natural $\mathbb{Z}_2$-grading, $E=E_0\oplus E_1$ where $E_0$ is the span of 1 and all monomials of even length (the centre of $E$), and $E_1$ is the span of all monomials of odd length (this is the anticommutative part of $E$). If $A=A_0\oplus A_1$ is a $\mathbb{Z}_2$-graded algebra its Grassmann envelope is $E(A)=E_0\otimes A_0\oplus E_1\otimes A_1$. Kemer proved that if $A$ is a PI algebra then it satisfies the same identities as the Grassmann envelope of a finitely generated $\mathbb{Z}_2$-graded algebra; moreover, "finitely generated" can be substituted by "finite dimensional". The theory developed by Kemer has a profound impact on PI theory. 

As already mentioned the relatively free algebra of $A$ is very important in studying the identities of $A$. But it is much larger than $A$, and may be difficult to study. That is why one wants to have a good model of the relatively free algebra. The most well known one is for the full matrix algebra $M_n(K)$, and it is called the generic matrix algebra. Assume $x_{ij}^k$ are independent commuting variables, $1\le i,j\le n$, $k=1$, 2, \dots{}, and form the matrices $X_k=(x_{ij}^k)$. These are matrices whose entries lie in the polynomial algebra $K[x_{ij}^k]$. The algebra $G_n$ generated by the generic matrices $X_k$ is the generic matrix algebra. It is isomorphic to the relatively free algebra for $M_n(K)$. It is well known it has no zero divisors, and its centre is a domain. Hence one may consider its field of fractions. An extremely important problem arises: determine the relationship between $K$ and the latter field. This was studied by Procesi in \cite{procesi1969}. Around 1950, Kaplansky posed the following problem: do there exist central polynomials for the matrix algebras $M_n(K)$? Recall that a central polynomial is a polynomial without a constant term in $K\langle X\rangle$ that assumes only central values when evaluated on $A$. Clearly the identities for $A$ are a particular case of central polynomials. 

Central polynomials for $M_n(K)$ were constructed independently by Formanek \cite{formanek1972central} and by Razmyslov \cite{razmyslov1973} (see also \cite{razmyslovbook}). It should be noted that a complete description of the central polynomials for $M_n(K)$ is known only when $n\le 2$. In fact, if $n>3$ not even the least degree of a central polynomial for $M_n(K)$ is known. 

When the algebra $A$ has an additional structure like a grading, or a trace, or an involution, it is natural to incorporate that structure into the identities of the algebra. In most of these cases such identities are "easier" to study; on the other hand, they provide very useful insights on the ideals of identities of algebras. 

In this paper we deal with algebras graded by (finite abelian) groups. Gradings on algebras are extremely important; they first appeared with the polynomial algebra in one or several variables; the grading on it is given by the degree. Later on the notion of a superalgebra became crucial in developing theories related to Mathematical Physics. Such structures have a $\mathbb{Z}_2$-grading plus additional restrictions. Recall that in the associative case no such additional restrictions are needed, though. The classification of the finite dimensional $\mathbb{Z}_2$-graded algebras that are simple as graded algebras (that is, they have no nontrivial homogeneous ideals) was given by C. T. C. Wall \cite{wall}. Group gradings on the matrix algebras $M_n(K)$ were completely described in \cite{bahturin2001group}. We refer the interested reader to the monograph \cite{elduque2013gradings} for more information concerning group gradings on associative and Lie algebras.

We study graded algebras $A=\oplus_{g\in G} A_g$ over an algebraically closed field of characteristic 0 assuming that the grading group $G$ is finite and abelian. A $G$-grading on $A$ is regular if for every $n$ and for every $n$-tuple of elements $g_1$, \dots, $g_n\in G$ there exist $a_i\in A_{g_i}$ such that $a_1\cdots a_n\ne 0$, and for every $g$, $h\in G$ there exists a scalar $\beta(g,h)\in K^*$ such that $a_ga_h=\beta(g,h)a_ha_g$ for every choice of $a_g\in A_g$ and $a_h\in A_h$. This notion was introduced and first studied by Regev and Seeman \cite{regevz2}, later on significant contribution to the topic was given in \cite{bcommutation, bahturin2009graded, Eli1}. See also \cite{LPK.1, LPK.2, L.P.K.3}. Observe that $\beta$ is a bicharacter on $G$. We study the graded central polynomials when the regular grading is minimal. Recall that a regular grading on $A$ is minimal whenever for every $g$, $h\in G$,  the equalities $\beta(g,x)=\beta(h,x)$ for every $x\in G$ imply $g=h$. If we form the matrix $M^A$ of order $|G|$ with entries $\beta(g,h)$, the grading is minimal if and only if $\det(M^A)\ne 0$. As shown in \cite{LPK.1} this does not necessarily hold if the characteristic of $K$ is $p>2$. 

We study the primeness property for the graded central polynomials of $A$. Recall that in \cite{regev1979primeness}, A. Regev proved that if $f$ and $g$ are polynomials in disjoint sets of variables such that $fg$ is a nonzero central polynomial for $M_n(K)$ then both $f$ and $g$ are central polynomials for $M_n(K)$. The algebras $M_n(E)$ and $M_{k,l}(E)$ also satisfy the primeness property for their central polynomials, see \cite{diniz2015primeness, DinizFidelis}. In the latter paper the graded version of the primeness property for central polynomials was studied as well. We first prove that if $\Gamma$ is a regular grading on $A$ then we can find a coarsening $\Gamma_0$ of $\Gamma$ such that the coarsening is a minimal regular grading. If one denotes by $G_0(\beta)=\{g\in G\mid \beta(g,h)=1 \text{ for each } h\in G\}$ then $G_0$ is a subgroup of $G$. We prove that the grading on $A$ induced by $G/G_0$ is regular and minimal. This allows us to deduce that $A$ satisfies the primeness property for graded central polynomials with respect to the grading $\Gamma$ if and only if it does with respect to $\Gamma_0$. We deduce that a $G$-graded regular algebra cannot satisfy the primeness property for graded central polynomials. As a consequence, the algebras $M_n(K)$ equipped with the Pauli grading, fail to satisfy the primeness property for graded central polynomials. Another consequence is a description of the situations where a grading on $M_n(K)$ satisfies the primeness property for graded central polynomials. 

We prove, in section 4 of the paper, that no nontrivial grading on $M_2(K)$ and on $M_3(K)$ satisfies the primeness property for graded central polynomials. It seems to us that the same holds for every $M_n(K)$, $n\ge 4$. 

Finally we consider $\mathbb{Z}_2$-graded regular algebras $A$ and the primeness property for their ordinary central polynomials. We use some ideas and results from \cite{L.P.K.3}, and we obtain that if the regular decomposition is minimal then $A$ satisfies the primeness property for the ordinary central polynomials. The key point here is that $A$ satisfies the same graded identities as $E$, moreover $A$ contains a copy of $E$. All this means that the infinite dimensional Grassmann algebra is the keystone in determining the $\mathbb{Z}_2$-graded regular minimal algebras.

\section{Background}

Throughout this paper, $K$ denotes an algebraically closed field of characteristic $0$ and $G$ is a finite abelian group (written additively). By a $G$-graded algebra $A$, we mean a $K$-algebra $A$ endowed with a decomposition into a direct sum of subspaces $A = \bigoplus_{g \in G} A_g$ satisfying $A_g A_h \subseteq A_{g + h}$ for all $g$, $h \in G$. In order to be precise with the grading, we will denote it by $\Gamma$, i.e.,
\[
\Gamma\colon A = \bigoplus_{g \in G} A_{g}.
\]
We shall write $(A, \Gamma)$ to denote the algebra $A$ jointly with its grading $\Gamma$. Consider two gradings $\Gamma\colon A=\bigoplus_{g\in G_{1}} A_{g}$ and $\Gamma'\colon A=\bigoplus_{h\in G_{2}} A_{h}'$ on $A$, where $G_{1}$ and $G_{2}$ are both finite groups. We say $\Gamma'$ is a \textsl{}{coarsening} of $\Gamma$ (and write $\Gamma'\leq \Gamma$) if for every $g\in G_{1}$ there exists $h\in G_{2}$ such that $A_{g}\subseteq A_{h}'$. Now, consider a $G_{1}$-graded algebra $\Gamma\colon A=\bigoplus_{g\in G_{1}} A_{g}$ and let $\alpha\colon G_{1}\rightarrow G_{2}$ be an epimorphism of groups. We can define a natural coarsening of $\Gamma$ induced by $\alpha$ as follows: for each $h\in G_{2}$, set $A_{h}':=\bigoplus_{g\in \alpha^{-1}(h)} A_{g}$. It can be quickly seen that this defines a grading on $A$, denoted by $\Gamma^{\alpha}$, and that $\Gamma^{\alpha}$ is a coarsening of $\Gamma$. This grading is called the \textsl{coarsening of $\Gamma$ realized by} $\alpha\colon G_{1}\rightarrow G_{2}$.

If $A$ is a $K$-algebra we shall denote its centre by $Z(A)$. Furthermore, if $A$ is endowed with a grading $\Gamma$, and we want to highlight its grading, we write $Z(A, \Gamma)$ instead of $Z(A)$. 

We denote $K\langle X\rangle$ the free unitary associative algebra freely generated by $X$ over $K$, where $X=\{x_{1},x_{2},\ldots,x_{n},\ldots\}$ is a countable infinite set of independent variables, and $K\langle X_{G}\rangle$ is the free unitary associative $G$-graded algebra freely generated by $X=\cup_{g\in G}X_g$ over $K$, where the $X_g$'s are infinite countable sets of disjoint variables. For a background on these basic concepts, we refer the reader to the monographs \cite{drensky2000free}, \cite{elduque2013gradings} and \cite{GZbook}.
 
\begin{Definition}
\label{Definition.1}
Let $A$ be a $K$-algebra endowed with a $G$-grading $\Gamma\colon A = \bigoplus_{g \in G} A_{g}$.

\begin{enumerate}
    \item A graded ideal $I'\trianglelefteq K\langle X_{G}\rangle$ is called $T_{G}$-\textsl{ideal} if $\phi(I')\subseteq I'$ for every graded endomorphism $\phi$ of $K\langle X_{G}\rangle$.
\item A polynomial $f=f(x_{1}^{(g_{1})},\ldots, x_{k_{1}}^{(g_{1})},x_{1}^{(g_{2})},\ldots, x_{k_{2}}^{(g_{2})},\ldots, x_{1}^{(g_{m})},\ldots, x_{k_{m}}^{(g_{m})})\in K\langle X_{G}\rangle$ is  a \textsl{graded polynomial identity} for the $G$-graded algebra $A$ if
\[
f(a_{1}^{(g_{1})},\ldots, a_{k_{1}}^{(g_{1})},a_{1}^{(g_{2})},\ldots, a_{k_{2}}^{(g_{2})},\ldots, a_{1}^{(g_{m})},\ldots, a_{k_{m}}^{(g_{m})})=0
\]    
for every $a_{1}^{(g_{i})}$, \dots, $a_{k_{i}}^{(g_{i})}\in A_{g_{i}}$, $1\leq i\leq m$. Moreover, the set of all graded identities for $A$, denoted by $T_{G}(A)$, is a $T_{G}$-ideal of $K\langle X_{G}\rangle$ called the \emph{$T_G$-ideal} of $A$.
\end{enumerate}
\end{Definition}
We recover the ordinary polynomial identities if $A$ is graded by the trivial group, and we denote by $T(A)$ the T-ideal of $A$.

\begin{Definition} 
\label{Definition.2}
Let $\Gamma\colon A=\bigoplus_{g\in G}A_{g}$ be $G$-graded. Given $f(x_{1}^{(g_{1})},\ldots, x_{n}^{(g_{n})})\in K\langle X_{G}\rangle$,  we say that $f$ is a \textsl{graded central polynomial} of $A$ if $f(a_{1},\ldots, a_{n}) \in Z(A)$ for all admissible substitution $a_{1}\in A_{g_{1}}$, \dots, $a_{n} \in A_{g_{n}}$. Furthermore, if $f(x_{1}^{(g_{1})},\ldots, x_{n}^{(g_{n})})\notin T_{G}(A)$, then, in this case, we refer to $f$ as a \textsl{proper graded central polynomial} of $A$.    
\end{Definition}

Given $\Gamma\colon A=\bigoplus_{g\in G} A_{g}$ a $G$-graded algebra,  we denote by $\mathscr{C}_{A,\Gamma}$ the vector space of all graded central polynomials and $\mathscr{X}_{A,\Gamma}$ the set of all proper graded central polynomials. Recall that $\mathscr{C}_{A,\Gamma}$ is invariant under all graded endomorphisms of $K\langle X_{G}\rangle$ (see \cite{DinizFidelis}).

\begin{Definition}
\label{primeness.definition}
Let $\Gamma\colon A=\bigoplus_{g\in G}A_{g}$ be a $G$-graded algebra. We say $A$ satisfies the \textsl{primeness property for graded central polynomials} if, given two graded polynomials in disjoint sets of variables $f(x_{1}^{(g_{1})},\ldots, x_{r}^{(g_{r})})$, $g(x_{r+1}^{(g_{r+1})},\ldots, x_{s}^{(g_{s})}) \in K\langle X_{G}\rangle$ satisfying
\[
f(x_{1}^{(g_{1})},\ldots, x_{r}^{(g_{r})})g(x_{r+1}^{(g_{r+1})},\ldots, x_{s}^{(g_{s})}) \in \mathscr{X}_{A,\Gamma},
\]
then both $f(x_{1}^{(g_{1})},\ldots, x_{r}^{(g_{r})}) \in \mathscr{X}_{A,\Gamma}$ and $g(x_{r+1}^{(g_{r+1})},\ldots, x_{s}^{(g_{s})}) \in \mathscr{X}_{A,\Gamma}$.
\end{Definition}

We recall the definition of elementary and fine gradings.

\begin{Definition}
\begin{enumerate}
    \item Let $n\in \mathbb{N}$ and let $N$ be a finite group. Consider the matrix algebra $C := M_{n}(K)$ and denote by $e_{i,j}$ the elementary matrix having $1$ in the $(i,j)$-th entry and $0$ elsewhere. Given an $n$-tuple $\overrightarrow{n}:=(g_{1},\ldots,g_{n}) \in N^{n}$, it defines an $N$-grading on $C$ by setting
\[
C_{g} := \operatorname{span}_{K} \{e_{i,j} \mid g_{j}g_{i}^{-1} = g\}
\]
for all $g \in N$.
This grading is called the \textsl{elementary grading} defined by $\overrightarrow{n}$.
\item Let $H$ be a group and $\Gamma\colon R=\bigoplus_{h\in H}R_{h}$ be an $H$-graded algebra. We say the grading $\Gamma$ is \textsl{fine} if $\dim R_{h}\leq 1$, for all $h\in G$. 
\end{enumerate}
\end{Definition}
\begin{Remark} In \cite[Theorem 5]{bahturin2001group}, Bahturin, Sehgal and Zaicev described all group gradings on the matrix algebra $M_n(K)$. In a sense, they proved that every group grading on $M_n(K)$ is a tensor product of an elementary and a fine grading (though by different groups). 
\end{Remark}

\begin{Example} Let $A$ be a PI algebra and let $T(A)$ be its $T$-ideal. We say $A$ is \emph{verbally prime} if whenever $f$ and $g$ are polynomials in disjoint sets of variables with $f$, $g\in T(A)$ and $fg\in T(A)$, then $f\in T(A)$ or $g\in T(A)$. 
Kemer showed in \cite{kemer1985varieties} that over a field of characteristic $0$, every verbally prime PI algebra is PI-equivalent to one of the following algebras $M_{n}(K)$ ($n\geq 1$), $M_{n}(E)$ (where $E$ is the infinite-dimensional Grassmann algebra), or $M_{k,l}(E)$ ($k$, $l> 0$), where $M_{k,l}(E)$ is a subalgebra of $M_{k+l}(E)$, which is defined in the following way: consider $E=E_{0}\oplus E_{1}$, the canonical $\mathbb{Z}_{2}$-grading  on $E$; an element $Y\in M_{k,l}(E) \subseteq M_{k+l}(E)$ is of the form
\[
Y = \begin{pmatrix}
    A & B \\
    C & D
\end{pmatrix},
\]
where $A \in M_{k}(E_0)$, $D \in M_{l}(E_0)$, and the blocks $B$ and $C$ have entries in $E_{1}$. 

The situation changes dramatically if $K$ is an infinite field of characteristic $p>0$. While the above T-ideals remain verbally prime, Razmyslov constructed various new T-ideals that are verbally prime, see \cite[Theorem 33.3, p.~157]{razmyslovbook}. A complete description of the verbally prime algebras in positive characteristic is not known. 

We recall several important results concerning central polynomials and the primeness property.

In \cite{regev1979primeness} Regev proved that $M_{n}(K)$ ($n\geq 1$) has the primeness property for central polynomials.

In \cite{diniz2015primeness}, Diniz obtained that $M_{n}(E)$ and $M_{n,n}(E)$ have the primeness property for central polynomials. Finally, in \cite{DinizFidelis} Diniz and Fidelis completed the study of the primeness property for verbally prime algebras by proving that $M_{k,l}(E)$ (for any $k$, $l>0$) also has the primeness property for central polynomials.

Let $Q$ be a finite group with $n:=|Q|$. Consider an elementary grading on $M_{n}(K)$ given by the $n$-tuple $(q_{1},\ldots,q_{n})$. This grading is called  \textsl{crossed product grading}
        if $q_{i}\neq q_{j}$ for any $i\neq j$. Suppose $Q$ has no non-trivial representation $\rho\colon Q\rightarrow K^{\ast}$.  Then, endowed with this grading $M_{n}(K)$ has the primeness property for graded central polynomials \cite[Theorem 3.10]{DinizFidelis}.
\end{Example}

\begin{Definition}
Let $\Gamma\colon A=\bigoplus_{g\in G}A_{g}$ be a $G$-graded algebra. We say that $A$ is \emph{$G$-graded strongly verbally prime} if, for any graded polynomials $f$, $g \in K\langle X_{G}\rangle$ in disjoint sets of variables, the condition $fg \in T_{G}(A)$ implies $f \in T_{G}(A)$ or $g \in T_{G}(A)$.
\end{Definition}
\begin{Example} Any $G$-graded division algebra is $G$-graded strongly verbally prime. In \cite[Theorem 1.14]{aljadeff2018verbally}, Aljadeff and Karasik gave a classification of $G$-graded strongly verbally prime algebras in terms of division algebras.
\end{Example}
\begin{Definition}
\label{def.regular.gradings}
Let $\Gamma\colon A=\bigoplus_{g\in G}A_{g}$ be a $G$-graded algebra. Then $A$ is a \emph{$G$-graded regular} algebra (or $A$ has a \emph{regular $G$-grading}), with regular decomposition, if the following conditions hold:
\begin{enumerate}
    \item[(i)] given $n\in \mathbb{N}$ and an $n$-tuple $(g_{1},\ldots, g_{n})\in G^{n}$, there exist homogeneous elements $a_{1}\in A_{g_{1}}$, \dots, $a_{n}\in A_{g_{n}}$ such that $a_{1}\cdots a_{n}\neq 0$;
    
     \item[(ii)] for every $g$, $h\in G$ and for every $a_{g}\in A_{g}$, $a_{h}\in A_{h}$, there exists $\beta(g,h)\in K^{\ast}$ satisfying
     \[
     a_{g}a_{h}=\beta(g,h)a_{h}a_{g}.
     \]
Here $\beta(g,h)$ depends only on $g$ and on $h$ but not on the choice of  $a_{g}\in A_{g}$, $a_{h}\in A_{h}$.
\end{enumerate}
Additionally, we define the \emph{regular decomposition matrix} associated to the regular decomposition of the $G$-graded algebra $A$ as $M^{A}=(\beta(g,h))_{g,h}$.
\end{Definition}

If $\Gamma\colon A=\bigoplus_{g\in G}A_{g}$ is a $G$-graded regular algebra, we say the regular decomposition of $A$ is \textsl{nonminimal} if there exist $g$, $h\in G$, $g\neq h$, such that $\beta(g,x)=\beta(h,x)$, for all $x\in G$. 

Otherwise the regular decomposition of $A$ is \textsl{minimal}. The regular decomposition of $A$ is minimal over an algebraically closed field of characteristic 0 if and only if $\det M^{A}\neq 0$, see \cite{Eli1}. The hypothesis of $K$ being of characteristic 0 cannot be removed from the last statement, because if $\operatorname{char}(K)>2$,  this result may fail as shown in \cite{LPK.1}.

The function $\beta\colon G\times G\rightarrow K^{\ast}$ is a \textsl{bicharacter} in the following sense: for every $g$, $h$ and $s\in G$ we have
\[
\beta(g,h)=\beta(h,g)^{-1},\quad \beta(g+s,h)=\beta(g,h)\beta(s,h), \text{ and }\beta(g,s+h)=\beta(g,s)\beta(g,h).
\]
Below we give a pair of examples of minimal regular decompositions.

\begin{enumerate}
    \item The Grassmann algebra $E$ with its canonical $\mathbb{Z}_{2}$-grading is a regular algebra whose regular decomposition  is minimal \cite[Example 6]{LPK.1}).
    \item Consider $M_{1,1}(E)$ and the matrices
    \[
    I=\begin{pmatrix}
        1 & 0\\
        0 &  1
    \end{pmatrix},\quad Y_{a}:=\begin{pmatrix}
        -1 & 0\\
        0 & 1
    \end{pmatrix},\quad Y_{b}:=\begin{pmatrix}
        0 & -1\\
        1 & 0
    \end{pmatrix}.
    \]
    Then, it can be seen that $M_{1,1}(E)=E_{0}I\oplus E_{0}Y_{a}\oplus E_{1}Y_{b}\oplus E_{1}Y_{a}Y_{b}$
    defines a structure of $\mathbb{Z}_{2}\times \mathbb{Z}_{2}$-graded regular algebra on $M_{1,1}(E)$ whose regular decomposition is minimal \cite[Example 2.1]{diniz2015primeness}.
\end{enumerate}
For further examples of regular gradings we refer the reader to \cite{LPK.1}.

\section{Primeness property of graded central polynomials for matrix algebras}

Given a bicharacter $\beta\colon G\times G\rightarrow K^{\ast}$, consider the following subgroup of $G$
\[
G_{0}(\beta)=\{g\in G\mid \beta(g,h)=1,\quad \text{for every}\quad h\in G\}.
\]
Let $\Gamma\colon A = \bigoplus_{g \in G} A_{g}$ be a $G$-graded regular algebra with bicharacter $\beta$. For every $g \in G$, denote by $\operatorname{Col}(g)$ the column $(\beta(g,x))_{x \in G}$ of the matrix $M^A$. 

\begin{Proposition}
\label{reduction}
Let $\Gamma$ be a regular grading on $A$. Then $A$ admits a regular grading $\Gamma_{0}$, which is a coarsening of $\Gamma$, such that the regular decomposition of $A$ with respect to $\Gamma_{0}$ is minimal, with bicharacter $\theta$ determined by $\beta$.
\end{Proposition}

\begin{proof}
Suppose the regular decomposition of $A$ is not minimal. Then there exist $g_{1}$, $g_{2} \in G$ such that $g_{1} \neq g_{2}$ and $\operatorname{Col}(g_{1}) = \operatorname{Col}(g_{2})$. Notice that 
\[
\beta(g_{1} - g_{2}, x) = \beta(g_{1}, x)\beta(-g_{2}, x) = \beta(g_{2}, x)\beta(-g_{2}, x)= 1,
\]
for every $x\in G$. This implies that $\operatorname{Col}(0) = \operatorname{Col}(g_{1} - g_{2})$. Therefore, $G_{0}(\beta) \neq \{0\}$. Now, consider the natural projection $\pi \colon G \rightarrow G / G_{0}(\beta)$, and observe that if $\overline{h} = h + G_{0}(\beta)$, then $g \in \pi^{-1}(\overline{h})$ if and only if $g - h \in G_{0}(\beta)$, which is equivalent to saying that
\[
\beta(g - h, x) = 1 \quad \text{for every } x \in G,
\]
or, in other words,
\[
\beta(g,x)=\beta(h,x) \quad \text{for all } x \in G.
\]
Consider the coarsening of $\Gamma$ realized by the epimorpism $\pi$
\[
\Gamma^{\pi}\colon A=\bigoplus_{\overline{h}\in G/G_{0}(\beta)} A_{\overline{h}}'
\]
where
\[
A_{\overline{h}}'=\bigoplus_{g\in \pi^{-1}(\overline{h})} A_{g}.
\]
As we have just seen, the component $A_{\overline{h}}'$ is a direct sum of homogeneous components $A_{g}$ satisfying $\beta(g,x)=\beta(h,x)$, for all $x\in G$. The grading $\Gamma^{\pi}$ is a coarsening of $\Gamma$.

Now we show that $\Gamma^{\pi}$ defines a regular grading. We set $T:=G/G_{0}(\beta)$ and let  $n\in \mathbb{N}$, and $(\overline{g}_{1},\ldots, \overline{g_{n}})\in T^{n}$. By definition, for every $1\leq i\leq n$, $A_{g_{i}}$ is a direct summand of $A_{\overline{g_{i}}}'$, and since $\Gamma$ is a regular grading, there exist $a_{i}\in A_{g_{i}}\subseteq A_{\overline{g_{i}}}'$ such that $a_{1}\cdots a_{n}\neq 0$. Finally, set
\[
\theta\colon T\times T\rightarrow K^{\ast},\quad (\overline{g},\overline{h})\mapsto \beta(g,h).
\]
In order to see that $\theta$ is a bicharacter of $T$ it is enough to show that $\theta$ is well-defined. Suppose $\overline{g}=\overline{g'}$ and $\overline{h}=\overline{h'}$; this means 
\[
\beta(g,x)=\beta(g',x),\quad \text{and}\quad \beta(h,x)=\beta(h',x),\quad \text{for every}\quad x\in G.
\]
Therefore
\[
\beta(g,h)=\beta(g',h)=(\beta(h,g'))^{-1}=(\beta(h',g'))^{-1}=\beta(g',h'),
\]
thus $\theta$ is well-defined. Now, given $a\in A_{\overline{g}}'$ and $b\in A_{\overline{h}'}$, there exist $a_{1}\in A_{g_{1}}$, \dots, $a_{k}\in A_{g_{k}}$, and $b_{1}\in A_{h_{1}}$, \dots, $b_{l}\in A_{h_{l}}$ with $g_{i}\in \pi^{-1}(\overline{g})$, for each $1\leq i\leq k$, and $h_{i}\in \pi^{-1}(\overline{h})$, for each $1\leq i\leq l$, such that $a=\sum_{i=1}^{k}a_{i}$ and $b=\sum_{j=1}^{l}b_{j}$. Thus we get
\begin{align*}
    ab &= \sum_{i}\sum_{j} a_{i}b_{j}=\sum_{j}\sum_{i}\beta(g_{i},h_{j})b_{j}a_{i} = \sum_{j}\sum_{i}\beta(g,h_{j})b_{j}a_{i}\\
    &=\sum_{j}\sum_{i}(\beta(h_{j},g))^{-1}b_{j}a_{i}
       = \sum_{j}\sum_{i}(\beta(h,g))^{-1}b_{j}a_{i}=\sum_{j}\sum_{i}\beta(g,h)b_{j}a_{i}\\
       &=\beta(g,h)ba
       = \theta(\overline{g},\overline{h})ba.
\end{align*}
This implies $A$ is a $T$-graded regular algebra with bicharacter $\theta$. Finally, if $\overline{g}\in T$ is such that $\theta(\overline{g},\overline{h})=1$, for all $\overline{h}\in T$, then, by definition we have
\[
\beta(g,h)=1,\quad \text{for every}\quad h\in G
\]
hence $g\in G_{0}(\beta)$. Therefore, $\overline{g}=0$, which implies $G_{0}(\theta)=\{0\}$. Setting $\Gamma_0$ the grading defined by $T$ we conclude that the regular decomposition of $A$ with respect to $\Gamma_{0}$ is minimal.
\end{proof}

\begin{Lemma}
\label{Key.remark}
Let $G$ and $H$ be finite abelian groups and let $\alpha \colon G \rightarrow H$ be an epimorphism. Let $A$ be an algebra, $\Gamma$ a $G$-grading of $A$ and $\Gamma'$ the $H$-grading of $A$ realized by $\alpha$ by means of coarsening.   

Then, a polynomial $t(x_{1}^{(h_{1})},\ldots, x_{r}^{(h_{r})})\in K\langle X_{H}\rangle$ is in $\mathscr{X}_{A,\Gamma'}$ if and only if $t_{G}:=t(\nu_{1},\ldots, \nu_{r})\in K\langle X_{G}\rangle$ is in $\mathscr{X}_{A,\Gamma}$, where $
\nu_{i}=\sum_{j_{i}=1}^{k_{i}} x_{i}^{(g_{j_{i},i})}$, and $\alpha^{-1}(h_{i})=\{g_{1,i},\ldots, g_{k_{i},i}\}$.
\end{Lemma}

\begin{proof}
If $t\in \mathscr{X}_{A,\Gamma'}$ there exist $a_{1}\in A_{h_{1}}'$,  \dots, $a_{r}\in A_{h_{r}}'$ such that $t(a_{1},\ldots, a_{r})\neq 0$. Writing
\[
 a_{i}=\sum_{j_{i}=1}^{k_{i}}a_{g_{j_{i},i}},\quad a_{g_{j_{i},i}}\in A_{g_{j_{i},i}}
\]
we have $t_{G}(a_{g_{1,1}},\ldots, a_{g_{k_{1},1}},\ldots, a_{g_{k_{r}},r})\neq 0$, that is $t_{G}\notin T_{G}(A)$. Given homogeneous elements $b_{g_{1,1}}\in A_{g_{1,1}}$, \dots, $b_{g_{k_{r}},r}\in A_{g_{k_r},r}$, put $b_{i}:= \sum_{j_{i}=1}^{k_{i}}b_{g_{j_{i},i}}$. 
Then $b_{i}\in A_{h_{i}}'$, for every $1\leq i\leq r$. Hence $t(b_{1},\ldots, b_{r})\in Z(A,\Gamma')$. On the other hand, $t_{G}(b_{g_{1,1}},\ldots, b_{g_{k_{r}},r})=t(b_{1},\ldots, b_{r})$, therefore $t_{G}(b_{g_{1,1}},\ldots, b_{g_{k_{r}},r})\in Z(A,\Gamma)$. 
Suppose $t_{G}\in \mathscr{X}_{A,\Gamma}$, then given homogeneous elements $a_{1}\in A_{h_{1}}'$,\dots, $a_{r}\in A_{h_{r}}'$, we write 
$a_{i}=\sum_{j_{i}=1}^{k_{i}}a_{g_{j_{i},i}}$ as above and we get 
\[
t(a_{1},\ldots, a_{r})=t_{G}(a_{g_{1,1}},\ldots, a_{g_{k_{r}},r}).
\]
We conclude $t\in \mathscr{X}_{A,\Gamma'}$. 
\end{proof}

\begin{Lemma}
\label{Lemma.aux.1}
Let $\Gamma \colon A = \bigoplus_{g \in G} A_g$ be a $G$-graded algebra, and let $\Gamma' \colon A = \bigoplus_{h \in H} A_{h}'$ be a coarsening of $\Gamma$ realized by an epimorphism $\alpha \colon G \rightarrow H$. If $(A,\Gamma)$ satisfies the primeness property for graded central polynomials, then so does $(A,\Gamma')$.
\end{Lemma}
\begin{proof}
Assume that $(A,\Gamma)$ satisfies the primeness property for graded central polynomials. Let 
\[
f(x_{1}^{(\mathbf{h}_{1})},\ldots, x_{r}^{(\mathbf{h}_{r})}) \quad \text{and} \quad g(x_{r+1}^{(\mathbf{h}_{r+1})},\ldots, x_{s}^{(\mathbf{h}_{s})})
\]
be polynomials of $K\langle X_{H}\rangle$ in disjoint sets of variables and suppose their product
\[
f(x_{1}^{(\mathbf{h}_{1})},\ldots, x_{r}^{(\mathbf{h}_{r})})g(x_{r+1}^{(\mathbf{h}_{r+1})},\ldots, x_{s}^{(\mathbf{h}_{s})})
\]
lies in $\mathscr{X}_{A,\Gamma'}$. Recall that
\[
A = \bigoplus_{\mathbf{h} \in H} A_{\mathbf{h}}', \quad \text{where} \quad A_{\mathbf{h}}' = \bigoplus_{g \in \alpha^{-1}(\mathbf{h})} A_g, \quad \text{for every } \mathbf{h} \in H.
\]

For each $i\in\{1,\ldots,s\}$, we choose $a_i \in A_{\mathbf{h}_i}'$. Since
\[
A_{\mathbf{h}_i}' = \bigoplus_{g \in \alpha^{-1}(\mathbf{h}_i)} A_g,
\]
there exist $g_{1,i}$, \dots,  $g_{k_i,i} \in \alpha^{-1}(\mathbf{h}_i)$ and $a_{g_{j,i}} \in A_{g_{j,i}}$ such that $a_{i} = \sum_{j} a_{g_{j,i}}$.

Consider the polynomials
\[
h(x_{1}^{(\mathbf{h}_1)}, \ldots, x_{s}^{(\mathbf{h}_s)}) = f(x_{1}^{(\mathbf{h}_1)}, \ldots, x_{r}^{(\mathbf{h}_r)}) g(x_{r+1}^{(\mathbf{h}_{r+1})}, \ldots, x_{s}^{(\mathbf{h}_s)})
\]
and
\[
h_{G} = f(\nu_1, \ldots, \nu_r)g(\nu_{r+1}, \ldots, \nu_s),
\]
as in Lemma \ref{Key.remark}, and notice that $h_G = f_G g_G$. Since $h\in \mathscr{X}_{A,\Gamma'}$, by Lemma \ref{Key.remark} we have $h_{G}\in \mathscr{X}_{A,\Gamma}$. By the primeness property of $(A, \Gamma)$, both $f_{G}$ and $g_{G}$ are proper graded central polynomials for $(A, \Gamma)$. Once again, by applying Lemma \ref{Key.remark}, it follows that $f$ and $g$ are proper graded central polynomials for $(A, \Gamma')$.
\end{proof}

We need the following technical results.
\begin{Lemma}
\label{equiv.to.multilinear}
Let $\Gamma\colon A=\bigoplus_{g\in G}A_{g}$ be a $G$-graded regular algebra with bicharacter $\beta$. Suppose $A$ has the primeness property for graded central multilinear polynomials. Then, $A$ has the primeness property for graded central polynomials. 
\end{Lemma}
\begin{proof} Indeed, by \cite[Proposition 4.10]{DinizFidelis}, $A$ is a $G$-graded strongly verbally prime algebra. Therefore, from \cite{DinizFidelis}, the proof is the same as  in \cite[Proposition 3.10]{diniz2015primeness}.
\end{proof}

\begin{Lemma} 
\label{Lemma.aux.2}
Let $\Gamma \colon A = \bigoplus_{g \in G} A_{g}$ be a $G$-graded regular algebra with bicharacter $\beta$, and let $\Gamma_{0}$ the minimal grading on $A$ realized by $\pi\colon G\rightarrow H:=G/G_{0}(\beta)$. Then, if $D(x_{1}^{(g_{1})},\ldots, x_{n}^{(g_{n})})\in K\langle X_{G}\rangle$ is a multilinear graded polynomial, we have
\[
D(x_{1}^{(g_{1})},\ldots, x_{n}^{(g_{n})})\in \mathscr{X}_{A,\Gamma},\quad \text{if and only if}\quad D(\pi)\in \mathscr{X}_{A,\Gamma_{0}}
\]
 where $D(\pi)\in K\langle X_{H}\rangle$ is defined by $D(\pi):=D(x_{1}^{(\pi(g_{1}))},\ldots, x_{n}^{(\pi(g_{n}))})$.
\end{Lemma}
\begin{proof} 
Modulo the \(G\)-graded identities of \(A\) (for more details on this \(T_G\)-ideal, we refer to \cite[Remark 3]{doi:10.1142/S0218196717500436}), we may assume that \(D\) is a multilinear monomial in \(K\langle X_G\rangle\), say \(D = x_{1}^{(g_{1})} \cdots x_{n}^{(g_{n})}\).

Suppose $D\in \mathscr{X}_{A,\Gamma}$. We need to show that $D(\pi)=x_{1}^{(\pi(g_{1}))} \cdots x_{n}^{(\pi((g_{n}))}\in \mathscr{X}_{A,\Gamma'}$. Set $h_{1}:=\pi(g_{1})$, \dots, $h_{n}:=\pi(g_{n})$ and for each $1\leq i\leq n$, consider the preimage $\pi^{-1}(h_{i})=\{g_{1,i},\ldots, g_{k_{i},i}\}$, with $g_{1,i}=g_{i}$. 
 In particular,
\[
\pi(g_{1,i})=\cdots=\pi( g_{k_{i},i})=h_{i},
\]
and by the definition of $\pi$, it follows that for every $1\leq j\neq j'\leq k_{i}$
\begin{equation}
\label{eq.aux.1}
    \beta(g_{j,i},h)=\beta(g_{j',i},h),\quad \text{for all}\quad h\in G.
\end{equation}
Since $D\in \mathscr{X}_{A,\Gamma}$, we conclude 
\begin{equation}
\label{eq.aux.1.1}
\beta(g_{1,1}+\cdots+g_{1,n},h)=1,\quad \text{for every}\quad h\in G.
\end{equation}
Now, given homogeneous elements $a_{1}\in A_{h_{1}}'$,\dots, $a_{n}\in A_{h_{n}}'$, write
\[
a_{i}:= \sum_{j_{i}=1}^{k_{i}}a_{g_{j_{i},i}}.
\]
Therefore,
\[
D(\pi)(a_{1},\ldots, a_{n})=\sum_{j_{1}}\cdots \sum_{j_{n}} a_{g_{j_{1},1}}\cdots a_{g_{j_{n},n}}.
\]
For each $n$-tuple $(j_{1},\ldots, j_{n})$ appearing in the sum above, we have the following
\begin{align*}
    \beta(g_{j_{1},1}+g_{j_{2},2}+\cdots+g_{j_{n},n},h) &= \beta(g_{j_{1},1},h)\cdots \beta(g_{j_{n},n},h)\\
    & \overset{=}{(\ref{eq.aux.1})}  \beta(g_{1,1},h)\cdots \beta(g_{1,n},h)\\
    &= \beta(g_{1,1}+\cdots+ g_{1,n},h)=1,\quad \text{for any}\quad h\in G.
\end{align*}
Hence, $D(\pi)$ lies $\mathscr{C}_{A,\Gamma_{0}}$. 

On the other hand, let $D(\pi)\in \mathscr{X}_{A,\Gamma_{0}}$. As above, for $a_{1}\in A_{g_{1}}$, \dots, $a_{n}\in A_{g_{n}}$, by construction,  $D(\pi)(a_{1},\ldots, a_{n})=D(a_{1},\ldots, a_{n})$, that is $D\in \mathscr{C}_{A,\Gamma}$.

Finally, \(D(\pi)\) and \(D\) are not graded identities for \((A,\Gamma')\) and \((A,\Gamma)\), respectively. This follows from the fact that they are multilinear graded monomials and that these gradings are \(G\) and $H$-regular. By definition, there are no graded monomials that are identities for such graded algebras. Hence, the result follows.
\end{proof}

Notice that the previous result slightly differs from Lemma \ref{Key.remark}. In that case, the result does not assert that every central proper polynomial in \(K\langle X_G\rangle\) follows from another central proper polynomial in \(K\langle X_H\rangle\), but rather establishes this only in a specific situation. However, the latter proposition extends this conclusion when a \(G\)-graded regular algebra is considered.

\begin{Proposition}
\label{Min.equiv}
Let $\Gamma \colon A = \bigoplus_{g \in G} A_{g}$ be a $G$-graded regular algebra. Then, $A$ satisfies the primeness property with respect to $\Gamma$ if and only if $A$ satisfies the primeness property with respect to the minimal grading $\Gamma_{0}$.
\end{Proposition}
\begin{proof} Let $H:=G/G_{0}(\beta)$, then by Lemma \ref{Lemma.aux.1}, we only need to show $(A,\Gamma)$ satisfies the primeness property assuming $(A,\Gamma_0)$ does. According to Lemma \ref{equiv.to.multilinear}, it is enough to show $(A,\Gamma)$ satisfies the primeness property for graded central multilinear polynomials. Let $f$ and $g\in K\langle X_{H}\rangle$ be graded multilinear polynomials in disjoint variables such that $fg\in \mathscr{X}_{A,\Gamma}$. Then, by Lemma \ref{Lemma.aux.2}, $f(\pi)g(\pi)\in \mathscr{X}_{A,\Gamma_{0}}$. Since  $f(\pi)$ and $g(\pi)$ have disjoint sets of variables, by the primeness property $f(\pi)\in \mathscr{X}_{A,\Gamma_{0}}$ and $g(\pi)\in \mathscr{X}_{A,\Gamma_{0}}$. Therefore, once again by Lemma \ref{Lemma.aux.2}, $f\in \mathscr{X}_{A,\Gamma}$ and $g\in \mathscr{X}_{A,\Gamma}$ and we are done.
\end{proof}

\begin{Theorem} Let $\Gamma\colon   A=\bigoplus_{g\in G}A_{g}$ be a $G$-graded regular algebra. Then $(A,\Gamma)$ does not satisfy the primeness property for graded central polynomials.  
\end{Theorem}
\begin{proof} 
By Proposition \ref{Min.equiv}, we assume that the $G$-grading on the $G$-graded regular algebra $A$ is minimal.  Write $G=\{0,g,g_{1},\ldots, g_{k}\}$, where $g\neq 0$. By regularity of $A$, there exist $a_{0}\in A_{0}$, $a_{g}\in A_{g}$, \dots, $a_{g_{k}}\in A_{g_{k}}$ such that $a_{0}a_{g}\cdots a_{g_{k}}\neq 0$. In particular, $a_{g}a_{0}\neq 0$ and $a_{g}a_{g_{i}}\neq 0$, for all $1\leq i\leq k$. If $A_{g}\subseteq Z(A)$, then $a_{g}a_{g_{i}}=a_{g_{i}}a_{g}$, for all $1\leq i\leq k$, yielding $\beta(g,h)=1$, for all $h\in G$ which is an absurd by the minimality of $A$. Hence, $A_{g}$ is not contained in $Z(A)$. Analogously one can see that $A_{-g}$ is not contained in $Z(A)$. Now, consider the monomials $f:=x_{1}^{(g)}$ and $\widetilde{f}:=x_{2}^{(-g)}$. Since $A_{g}$ and $A_{-g}$ are not contained in $Z(A)$, it follows that $f$, $\widetilde{f}\notin \mathscr{X}_{A,\Gamma}.$ But $f\widetilde{f}\in \mathscr{X}_{A,\Gamma}$. 
\end{proof}
\begin{Corollary} 
\label{Pauli.grading}
For every $n\geq 1$, the algebra $M_{n}(K)$ endowed with Pauli's grading does not satisfy the primeness property for graded central polynomials. 
\end{Corollary} 
\begin{Corollary}
\label{aux.lemma.1}
Let $G$ be a finite abelian group, $Q$ a non-trivial subgroup of $G$, $\alpha\in H^{2}(Q,K^{\ast})$ a cocycle such that the bicharacter induced by $\alpha$ is non-trivial, $r> 1$, and $(g_{1},\ldots, g_{r})\in G^{r}$ an $r$-tuple. Consider $(K^{\alpha}Q\otimes M_{r}(K),\Gamma)$,  where the grading $\Gamma$ is defined by $(K^{\alpha}Q\otimes M_{r}(K))_{g}=\text{span}_{K}\{X_{h}\otimes e_{i,j}\mid g=-g_{j}+h+g_{i}\}$, $g\in G$, and $\{X_{q}\mid q\in Q\}$ denotes the canonical basis of $K^{\alpha}Q$ as well as $e_{i,j}$ is the $(i,j)$-elementary matrix. Then, $(K^{\alpha}Q\otimes M_{r}(K),\Gamma)$ does not satisfy the primeness property for graded central polynomials.
\end{Corollary}

\begin{proof}  Let $\beta_{\alpha}$ be the bicharacter induced by $\alpha$, $\beta_{\alpha}(g,h)=\alpha(g,h)\alpha^{-1}(h,g)$. Since $\beta_{\alpha}$ is non-trivial, there exist $q\in Q$ and $q'\in Q$ such that  $\beta_{\alpha}(q,q')\neq 1$.  Consider the following polynomials in $K\langle X_{G}\rangle$
\[
f(x_{1}^{(q)})=x_{1}^{(q)}, \quad \widetilde{f}(x_{2}^{(-q)})=\alpha(q'-q)^{-1}x_{2}^{(-q)}, \text{ and } h(x_{1}^{(q)},x_{2}^{(-q)})=f(x_{1}^{(q)})\widetilde{f}(x_{2}^{(-q)}).
\]
 If $I_{r}$ is the identity matrix in $M_{r}(K)$, then  $X_{q}\otimes I_{r}\notin Z( K^{\alpha}Q\otimes M_{r}(K))$, because
 \[
 (X_{q}\otimes I_{r})(X_{q'}\otimes I_{r})=\beta_{\alpha}(q,q')(X_{q'}\otimes I_{r})(X_{q}\otimes I_{r})
 \]
 and $\beta_{\alpha}(q,q')\neq 1$. 
 On the other hand,
\[
h(X_{q}\otimes I_{r}, X_{-q}\otimes I_{r})=X_{0}\otimes I_{r}\in Z(K^{\alpha}Q\otimes M_{r}(K)).
\]
 This implies that $K^{\alpha}Q\otimes M_{r}(K)$ has not the primeness property for graded central polynomials. 
\end{proof}

We recall the next result from \cite{DinizFidelis} regarding the primeness property for graded central polynomials. 

\begin{Theorem}
\label{Diniz.Claudemir.Theorem}
\cite[Theorem 3.10]{DinizFidelis} Let $M_{n}(K)$ be endowed with an elementary grading given by an $n$-tuple of pairwise distinct elements of $G$. Then $M_n(K)$ has the primeness property for graded central polynomials if and only if the grading is a crossed product grading and the group $G$ has no non-trivial homomorphism $G\rightarrow K^{\ast}$.
\end{Theorem}

Recall that an elementary grading by an $n$-tuple of pairwise distinct elements by a group $G$ of order $n$ is called a \textsl{crossed product grading} (see \cite{aljadeff2013crossed}).

Observe that the previous result only concerns elementary gradings defined by an $n$-tuple of pairwise distinct elements. But what about the case when the elementary grading is induced by an $n$-tuple where some of the elements may repeat? To shed light on such an issue, we study the general case in the next results.


Let $M_{n}(K)$ be endowed with a grading $\Gamma$, given by the finite abelian group $G$. By \cite[Theorem 3]{bahturin2005finite}, $(M_{n}(K),\Gamma)$ is graded isomorphic to $(K^{\alpha}Q\otimes M_{r}(K),\Gamma')$, where $r\geq 1$, $Q$ is a non-trivial subgroup of $G$, $r=(g_{1},\ldots,g_{r})\in G^{r}$,  $\alpha\in H^{2}(Q,K^{\ast})$, and  $\Gamma'$ is defined by $(K^{\alpha}Q\otimes M_{r}(K))_{g}=\text{span}_{K}\{X_{h}\otimes e_{i,j}\mid g=-g_{j}+h+g_{i}\}$, $g\in G$. 

\begin{Theorem} 
\label{elementary}
Let $(M_{n}(K),\Gamma)$ be a matrix algebra over an algebraically closed field $K$ of characteristic zero, endowed with a grading $\Gamma$ by a group $G$. Under these conditions, $(M_n(K),\Gamma)$ is isomorphic to the tensor product $(K^{\alpha}Q \otimes M_r(K), \Gamma')$, where $M_r(K)$ carries an elementary grading and $K^{\alpha}Q$ denotes the twisted group algebra determined by the cocycle $\alpha$.
    \begin{enumerate}
        \item[$(i)$] 
        Suppose the bicharacter induced by $\alpha$ is non-trivial. If $(M_{n}(K),\Gamma)$ satisfies the primeness property for graded central polynomials, then $\Gamma$ is an elementary grading.
        \item [$(ii)$] 
        Suppose the bicharacter induced by $\alpha$ is trivial. Then $(M_{n}(K),\Gamma)$ satisfies the primeness property for graded central polynomials if and only if $M_{r}(K)$ satisfies it, endowed with its elementary grading induced by $r$.
    \end{enumerate}
\end{Theorem}
\begin{proof} For $(i)$, suppose  $(M_{n}(K),\Gamma)$ has the primeness property for graded central polynomials, and the bicharacter induced by $\alpha$ is non-trivial. 

If $r=1$, then $(M_{n}(K),\Gamma)$ is the Pauli's grading but this is not allowed according to Corollary \ref{Pauli.grading}. On the other hand, if $Q\neq \{0\}$, then by Corollary \ref{aux.lemma.1} $(M_{n}(K),\Gamma)$ does not satisfy the primeness property for graded central polynomials. We get $Q=\{0\}$, $r=n$ and $\Gamma$ is an elementary grading of $M_{n}(K)$.

Let us prove $(ii)$. Suppose the bicharacter induced by $\alpha$ is trivial. Then, $K^{\alpha}Q$ is commutative and by \cite[Corollary 2.2.5]{karpilovsky1993group}, $K^{\alpha}Q$ is isomorphic to the group algebra $KQ$. Then \( KQ \simeq KQ \cdot I_n \), where \( I_n \) denotes the identity matrix in \( M_n(K) \). Hence, we may 
assume that $KQ \subseteq Z(M_n(K))=K$. Hence $KQ=K$, and $Q=\{0\}$. In this way, \( \Gamma \) comes from an elementary grading induced by the tuple \( r \). Therefore, \( (M_{n}(K), \Gamma) \) has the primeness property if and only if \( M_{r}(K) \) has the primeness property for graded central polynomials.
\end{proof}

\section{\texorpdfstring{Graded central polynomials in $M_{2}(K)$ and $M_{3}(K)$ }{Primeness property for graded central polynomials in }}

In this section we study the proper central primeness in the case of matrix algebras of order 2 and 3, endowed with an elementary grading. 

Throughout, we fix a finite abelian group $G$. We consider $R:=M_{3}(K)$ endowed with the elementary grading $\Gamma$ given by the triple $(g_{1},g_{2},g_{3})$, where $g_{1}=g_{2}=g$ and $g_{3}\neq g$.  We know that
	\[
	R_{0}=\operatorname{span}_{K}\{e_{1,1},e_{2,2},e_{3,3},e_{1,2},e_{2,1}\},\quad R_{g_{3}-g}=\operatorname{span}_{K}\{e_{1,3},e_{2,3}\}
	\]
	and
	\[
	R_{g-g_{3}}=\operatorname{span}_{K}\{e_{3,1},e_{3,2}\}.
	\]

First we prove that there is no multilinear polynomial of degree $3$ in $\mathscr{X}_{R,\Gamma}$. Indeed, suppose there exists a multilinear polynomial $f\in \mathscr{X}_{R,\Gamma}$ such that $\deg_{\mathbb{Z}}f=3$. Therefore, we can write
		\[
		f(x_{1}^{(h_{1})},x_{2}^{(h_2)},x_{3}^{(h_{3})})=\sum_{\sigma\in S_{3}}\gamma_{\sigma}x_{\sigma(1)}^{(h_{\sigma(1)})}x_{\sigma(2)}^{(h_{\sigma(2)})}x_{\sigma(3)}^{(h_{\sigma(3)})},
		\]
		where $(h_{1},h_{2},h_{3})\in (\operatorname{Supp}(R))^{3}$. Denote by $I_{3}$ the identity matrix in $R$; since $f\in \mathscr{X}_{R,\Gamma}$ and $f$ is multilinear, we have  $f(X_{1},X_{2},X_{3})=cI_{3}$, with $c\neq 0$ if $X_{1}:=e_{s_{1},s_{2}}\in R_{h_{1}}$, $X_{2}:=e_{s_{3},s_{4}}\in R_{h_{2}}$ and $X_{3}:=e_{s_{5},s_{6}}$. Given $\sigma\in S_{3}$, set $M_{\sigma}:=x_{\sigma(1)}^{(h_{\sigma(1)})}x_{\sigma(2)}^{(h_{\sigma(2)})}x_{\sigma(3)}^{(h_{\sigma(3)})}$. Up to a permutation, we assume $M_{1}(X_{1},X_{2},X_{3})\neq 0$. Now, $X_{1}X_{2}X_{3}=e_{s_{1},s_{2}}e_{s_{3},s_{4}}e_{s_{5},s_{6}}\neq 0$, if and only if $s_{2}=s_{3}=s$ and $s_{4}=s_{5}=r$. Because $f(X_{1},X_{2},X_{3})$ is a diagonal matrix, we get $s_{1}=s_{6}=t$. Observe the following table:
		
        \begin{table}[h] 
			\centering      
			\begin{tabular}{|c|} 
				\hline
				$M_{1}(X_{1},X_{2},X_{3})=e_{t,t} $       \\
				$M_{(1,2,3)}(X_{1},X_{2},X_{3})=e_{s,s} $           \\
				$M_{(1,3,2)}(X_{1},X_{2},X_{3})=e_{r,r} $ \\
				$M_{(1,2)}(X_{1},X_{2},X_{3})=\delta_{r,t}\delta_{s,r}e_{s,t}$ \\
				$M_{(1,3)}(X_{1},X_{2},X_{3})=\delta_{t,s}\delta_{r,t}e_{r,s}$\\
				$M_{(2,3)}(X_{1},X_{2},X_{3})=\delta_{s,r}\delta_{t,s} e_{t,r}$\\
				\hline
			\end{tabular}
			\label{tab:duas-colunas}
		\end{table}
        \noindent
        where $\delta_{i,j}$ denotes the Kronecker delta.
		By hypothesis, if $\mathcal{T}_{3}\subseteq S_{3}$ denotes the set of all transpositions in $S_{3}$, we have
        \[
        \sum_{\sigma\in \mathcal{T}_{3}} \gamma_{\sigma} M_{\sigma}(X_{1},X_{2},X_{3})=0.
        \]
         Let $\{t,s,r\}= \{1,2,3\}$ and suppose $s=3$ (the remaining cases are similar). Then, we have $X_{1}=e_{t,3}$, $X_{2}=e_{3,r}$. In this case, since not all $h_{i}$ are zero, $h_{1}+h_{2}+h_{3}=0$, and the only element in $R_{0}$ that has the form $e_{\ast,3}$ or $e_{3,\ast}$ is $e_{3,3}$, we conclude $X_{1}\in R_{g_{3}-g}$, $X_{2}\in R_{g-g_{3}}$ and $X_{3}\in R_{0}$. Furthermore, $\{t,r\}=\{1,2\}$. Without loss of generality we assume $t=1$ and $r=2$. As a result we have
		\[
		f(X_{1},X_{2},X_{3})=\gamma_{1}e_{1,1}+\gamma_{(1,2,3)}e_{3,3}+\gamma_{(1,3,2)}e_{2,2}=cI_{3},
		\]
		which implies $\gamma_{1}=\gamma_{(1,2,3)}=\gamma_{(1,3,2)}=c$.
		
		Now take $Z_{1}:=e_{1,3}$, $Z_{2}=e_{3,1}$ and $Z_{3}:=e_{1,1}$. We have 
	\[
	f(Z_{1},Z_{2},Z_{3})=(\gamma_{1}+\gamma_{(1,3,2)})e_{1,1}+\gamma_{(1,2,3)}e_{3,3}=2ce_{1,1}+ce_{3,3}\neq 0
	\]
	which is a contradiction because it is not central. 

	\begin{Remark} If $a$, $b\in R_{h}$, with $h\in \{g-g_{3},g_{3}-g\}$, then by definition $ab=0$. In particular, if $x_{1}^{(h_{1})}\cdots x_{n}^{(h_{n})}\in\mathscr{X}_{R,\Gamma}$, with $h_{i}\neq 0$, for all $1\leq i\leq n$, we conclude
		\[
		(h_{1},\ldots, h_{n})\in \{(g-g_{3},g_{3}-g\ldots, g_{3}-g), (g_{3}-g,g-g_{3},\ldots, g_{3}-g)\}.
		\]
		If $(h_{1},\ldots, h_{n})=(g_{3}-g,g-g_{3},\ldots, g_{3}-g)$, then we would have
		\[
		h_{1}+\cdots+h_{n}=(g_{3}-g)+(g-g_{3})+\cdots+(g_{3}-g)+(g-g_{3})+(g_{3}-g)=g_{3}-g
		\]
		that is a contradiction because $g\neq g_{3}$. Thus $(h_{1},\ldots, h_{n})=(g-g_{3},g_{3}-g\ldots, g_{3}-g)$. 
\end{Remark}

\begin{Lemma} 
\label{conjugation} Let $S\in R_{0}$ be a diagonalizable matrix. Then, there exists an invertible matrix $E\in R_{0}$ such that $ESE^{-1}=D$, where $D$ is a diagonal matrix. In particular, for every $h\in \operatorname{Supp}(R)$ and invertible $Y\in R_{h}$, it follows that $EYE^{-1}\in R_{h}$.
\end{Lemma}
\begin{proof} Since $R_{0} = \operatorname{span}_{K}\{e_{1,1}, e_{2,2}, e_{3,3}, e_{1,2}, e_{2,1}\}$, the matrix $S$ has the form
\[
S = \begin{pmatrix}
    S' & 0 \\
    0 & d
\end{pmatrix},
\]
where $S' \in M_{2}(K)$ and $d \in K$. Denote by $m_{S}(x)$ the minimal polynomial of $S$, and by $m_{S'}(x)$ the minimal polynomial of $S' \in M_{2}(K)$. Of course $(x-d)$ is the minimal polynomial of the $1\times 1$ block $[d]$. 

Since $S$ is diagonalizable, by \cite[Theorem 6.6]{hoffmann1971linear}, the irreducible factors of $m_{S}$ must be linear with multiplicity $1$. Hence, there are exactly three possible cases for $m_{S}(x)$:
\[
m_{S}(x) = (x-a)(x-b)(x-d), \quad m_{S}(x) = (x-a)(x-d), \quad \text{or} \quad m_{S}(x) = (x-d),
\]
where $a$, $b$, $d$ and $e$ are pairwise distinct elements in $K$.

On the other hand, it is well known that both $m_{S'}(x)$ and $(x-d)$ must divide $m_{S}(x)$. If $m_{S}(x) = (x-d)$, then $S'$ is conjugate (hence equal) to $\begin{pmatrix} d & 0 \\ 0 & d \end{pmatrix}$. If $m_{S'}(x) = (x-a)(x-b)$ or $m_{S'}(x) = (x-a)$, then in both cases, again by \cite[Theorem 6.6]{hoffmann1971linear}, $S'$ is diagonalizable. In particular, there exists an invertible matrix $P \in M_{2}(K)$ such that $PS'P^{-1}$ is diagonal. Now, define
$E := \begin{pmatrix}
    P & 0 \\
    0 & 1
\end{pmatrix}$, then it is invertible because
\[
E\begin{pmatrix}
    P^{-1} & 0 \\
    0 & 1
\end{pmatrix}=\begin{pmatrix}
    1 & 0 & 0\\
    0 & 1 & 0\\
    0 & 0 & 1
    
\end{pmatrix}.
\]
Hence, $ESE^{-1}=D$, where $D=\operatorname{diag}(\lambda_{1},\lambda_{2},\lambda_{3})$ and $\lambda_{1}$,$\lambda_{2}$ and $\lambda_{3}$ are the eigenvalues of $S$. In particular, for any $h \in \operatorname{Supp}(R)$, we have $E R_{h} E^{-1} \subseteq R_{h}$, and the proof is complete.
\end{proof}
\begin{Proposition}
\label{main.prop.matrix.algebra}
The graded algebra $(R,\Gamma)$ does not satisfy the primeness property for graded central polynomials.
\end{Proposition}
\begin{proof}
    We argue similarly to \cite{formanek1972central}. Consider the free commutative algebra in four generators $K[z_{1},z_{2},z_{3},z_{4}]$. If $(h_{1},h_{2},h_{3})=(g_{3}-g,g-g_{3},0)$, we can define a linear function $\iota\colon K[z_{1},z_{2},z_{3},z_{4}]\rightarrow K\langle x_{0}^{(0)},x_{1}^{(h_{1})},x_{2}^{(h_{2})},x_{3}^{(h_{3})}\rangle$ in the following way 
\[
\iota\big(\sum c_{a}z_{1}^{a_{1}}z_{2}^{a_{2}}z_{3}^{a_{3}}z_{4}^{a_{4}}\big)= \sum c_{a} (x_{0}^{(0)})^{a_{1}}x_{1}^{(h_{1})}(x_{0}^{(0)})^{a_{2}}x_{2}^{(h_{2})}(x_{0}^{(0)})^{a_{3}}x_{3}^{(h_{3})}(x_{0}^{(0)})^{a_{4}}.
\]

 Consider the following polynomial in $K[z_{1},z_{2},z_{3},z_{4}]$:

\begin{align*}
	L(z_{1},z_{2},z_{3},z_{4})=\prod_{2\leq i\leq 3}(z_{1}-z_{i})(z_{4}-z_{i})\prod_{2\leq j<k\leq 3}(z_{j}-z_{k})^2
\end{align*}

Let $Z=c_{1}e_{1,1}+c_{2}e_{2,2}+c_{3}e_{3,3}$ be a diagonal matrix. If $(i_{1},i_{2},i_{3})$ is a permutation of $\{1,2,3\}$, it follows that 
\[
L(c_{i_{1}},c_{i_{2}},c_{i_{3}},c_{i_{1}})=(c_{1}-c_{2})^{2}(c_{1}-c_{3})^{2}(c_{2}-c_{3})^{2}=D(c_1,c_2,c_3)
\]
is the discriminant of $c_{1},c_{2},c_{3}$, which, by abuse of notation, will be called \textsl{discriminant} of $Z$; we denote it by $D(Z)$. Now, consider $X_{1}:=e_{i,3}$, $X_{2}:=e_{3,j}$, $X_{3}:=e_{k,l}$ and set $U:=\iota(L)$. For each $\gamma\in S_{3}$ define 
\[
M_{\gamma}(x_{1}^{(h_{1})},x_{2}^{(h_{2})},x_{3}^{(h_{3})}):=x_{\gamma(1)}^{(h_{1})}x_{\gamma(2)}^{(h_{2})}x_{\gamma(3)}^{(h_{3})}.
\]
We notice that for every $a_{1}$, $a_{2}$, $a_{3}$, $a_{4}\in \mathbb{N}$, we have
\[
Z^{a_{1}}X_{1}Z^{a_{2}}X_{2}Z^{a_{3}}X_{3}Z^{a_{4}}=c_{i}^{a_{1}}c_{3}^{a_{2}}c_{k}^{a_{3}}c_{l}^{a_{4}}M_{1}(X_{1},X_{2},X_{3}),
\]
therefore $U(Z,X_{1},X_{2},X_{3})=L(c_{i},c_{3},c_{k},c_{l})M_{1}(X_{1},X_{2},X_{3})$. On the other hand, $M_{1}(X_{1},X_{2},X_{3})=e_{i,3}e_{3,j}e_{k,l}\neq 0$ if and only if $j=k$, and $L(c_{i},c_{3},c_{k},c_{l})\neq 0$ if and only if $(i,3,k)$ is a permutation of $\{1,2,3\}$ and $i=l$.

Now, we have
\[
U(Z,X_{2},X_{3},X_{1})=L(c_{3},c_{k},c_{i},c_{3})M_{(1,2,3)}(X_{1},X_{2},X_{3})
\]
where $L(c_{3},c_{k},c_{i},c_{3})\neq 0$ if and only if $(i,k)$ is a permutation of $\{1,2\}$ and, moreover $M_{(1,2,3)}(X_{1},X_{2},X_{3})\neq 0$, if and only if $j=k$ and $l=i$. 

Finally, in a similar way we get
\[
U(Z,X_{3},X_{1},X_{2})=L(c_{k},c_{i},c_{3},c_{j})M_{(1,3,2)}(X_{1},X_{2},X_{3}).
\]
In this last case, $L(c_{k},c_{i},c_{3},c_{j})\neq 0$, if and only if $(k,i,3)$ is a permutation of $\{1,2,3\}$ and $j=k$. Moreover $M_{(1,3,2)}(X_{1},X_{2},X_{3})\neq 0$ if and only if $l=i$.

It follows $U(Z,X_{1},X_{2},X_{3})=0$ if and only if $U(Z,X_{2},X_{3},X_{1})=0$, if and only if  $U(Z,X_{3},X_{1},X_{2})=0$. When one among $U(Z,X_1,X_2,X_3)$, $U(Z,X_2,X_3,X_1)$, $U(Z,X_3,X_1,X_2)$ is non-zero we get
\begin{align}
\label{eq.key.2}
U(Z,X_1,X_2,X_3) &= D(Z)e_{i,i}, \nonumber\\
U(Z,X_2,X_3,X_1) &= D(Z)e_{3,3}, \\
U(Z,X_3,X_1,X_2) &= D(Z)e_{j,j}. \nonumber
\end{align}

On the other hand, suppose $Z=e_{1,2}$. Then, if $\sigma=(1,2,3)$ and $1\leq t\leq 3$ in all products of the form $Z^{a_{1}}X_{\sigma^{t}(1)}Z^{a_{2}}X_{\sigma^{t}(2)}Z^{a_{3}}X_{\sigma^{t}(k)}Z^{a_{4}}$ we will always find a term of the form $Ze_{3,\star}$. Hence $Ze_{3,\star}=0$ and $Z^{a_{1}}X_{\sigma^{t}(1)}Z^{a_{2}}X_{\sigma^{t}(2)}Z^{a_{3}}X_{\sigma^{t}(k)}Z^{a_{4}}=0$. Similarly for $Z=e_{2,1}$. 

Since the expressions $U(x_{0}^{(0)},x_{1}^{(h_{1})},x_{2}^{(h_{2})},x_{3}^{(h_{3})})$,  $U(x_{0}^{(0)},x_{3}^{(h_{3})},x_{1}^{(h_{1})},x_{2}^{(h_{2})})$ and $U(x_{0}^{(0)},x_{2}^{(h_{2})},x_{3}^{(h_{3})},x_{1}^{(h_{1})})$ are linear in $x_{1}^{(h_{1})}$, $x_{2}^{(h_{2})}$, $x_{3}^{(h_{3})} $ we conclude that if $Z$ is a diagonal matrix, then $U(Z,X_{1},X_{2},X_{3})$, $U(Z,X_{2},X_{3},X_{1})$ and $U(Z,X_{3},X_{1},X_{2})$ have the form established in (\ref{eq.key.2}), for all $X_{1}\in R_{h_{1}}$, $X_{2}\in R_{h_{2}}$ and $X_{3}\in R_{h_{3}}$.

Consider the polynomial
\begin{align*}
	P(x_{0}^{(0)},x_{1}^{(h_{1})},x_{2}^{(h_{2})},x_{3}^{(h_{3})}) &:= U(x_{0}^{(0)},x_{1}^{(h_{1})},x_{2}^{(h_{2})},x_{3}^{(h_{3})})+\\
	&+U(x_{0}^{(0)},x_{3}^{(h_{3})},x_{1}^{(h_{1})},x_{2}^{(h_{2})})-U(x_{0}^{(0)},x_{2}^{(h_{2})},x_{3}^{(h_{3})},x_{1}^{(h_{1})}).
\end{align*}

Since the polynomials $U(Z,X_1,X_2,X_3)$, $U(Z,X_2,X_3,X_1) $ and $U(Z,X_3,X_1,X_2)$ vanish for $Z=e_{1,2}$ or $Z=e_{2,1}$, we conclude that $P(Z,X_{1},X_{2},X_{3})$ also vanishes in this case. Moreover, for every diagonal matrix $Z$, and $X_{i}\in R_{h_{i}}$, we get
\begin{equation}
\label{eq.key.3}
    P(Z,X_{1},X_{2},X_{3})=\begin{pmatrix}
        D(Z) & 0 & 0\\
        0 & D(Z) & 0\\
        0 & 0 & -D(Z)
    \end{pmatrix}.
\end{equation}

By Lemma \ref{conjugation}, the components $R_{0}$, $R_{g_{3}-g}$, and $R_{g-g_{3}}$ are invariant under conjugation by an invertible matrix in $R_{0}$. Furthermore, given $Z'\in R_{0}$ a diagonalizable matrix, again by Lemma \ref{conjugation}, there exists an invertible matrix $E=\begin{pmatrix}
    S & 0\\
    0 & e
\end{pmatrix}\in R_{0}$, such that $EZ'E^{-1}=Z$, where $Z\in R_{0}$ is a diagonal matrix. Hence,  given $X_{i}\in R_{h_{i}}$, $1\leq i\leq 3$, we get
\begin{align*}
EP(Z,X_{1},X_{2},X_{3})E^{-1 }&=\begin{pmatrix}
    S & 0\\
    0 & e
\end{pmatrix}\begin{pmatrix}
        D(Z) & 0 & 0\\
        0 & D(Z) & 0\\
        0 & 0 & -D(Z)
        \end{pmatrix}\begin{pmatrix}
    S^{-1} & 0\\
    0 & e^{-1}
\end{pmatrix}\\
&=P(Z,X_{1},X_{2},X_{3})
\end{align*}
implying that for every diagonalizable matrix $Z'\in R_{0}$ and $X_{i}\in R_{h_{i}}$, $1\leq i\leq 3$, the evaluation $P(Z',X_{1},X_{2},X_{3})$ is of the form given in (\ref{eq.key.3}).

We only need to show that these polynomials have the form described above, for every $Z \in R_{0}$.

 Without loss of generality, we can consider $Z\in R_{0}$ as a generic matrix whose entries are independent commuting indeterminates (in this case such a matrix is diagonalizable by \cite[Lemma 7.2.5]{drensky2000free}). Then, the fact that $P(Z,X_1,X_2,X_3)$ has the form (\ref{eq.key.3}) for every $Z \in R_{0}$ and  $X_{i}\in R_{h_{i}}$, $1\leq i\leq 3$, follows by specialization, as in \cite[Theorem 
 7.3.5]{drensky2000free} or \cite{formanek1972central}.

Since $K$ is an algebraically closed field of characteristic $0$, there exists a diagonal matrix $Z$ with distinct eigenvalues, so $D(Z)\neq 0$. Hence, if in the elements $X_{1}=e_{i,3}$, $X_{2}=e_{3,j}$ and $X_{3}=e_{k,l}$ we have $j=k=2$ and $l=i=1$, we get
\[
P(Z,X_{1},X_{2},X_{3})=\begin{pmatrix}
	D(Z) & 0 & 0\\
	0 & D(Z) & 0\\
	0 & 0 & -D(Z)
\end{pmatrix}\neq 0.
\]
In particular, $P\in T_{G}(R)$. Moreover, by construction, for all $Z\in R_{0}$, $X_{1}\in R_{h_{1}}$, $X_{2}\in R_{h_{2}}$ and $X_{3}\in R_{h_{3}}$ such that $P(Z,X_{1},X_{2},X_{3})\neq 0$, we have
\[
P(Z,X_{1},X_{2},X_{3})\in K\begin{pmatrix}
	1 & 0 & 0\\
	0 & 1 & 0\\
	0 & 0 & -1
\end{pmatrix}.
\]
Now, consider $P'$ being a copy of $P$ in distinct variables. Of course, $PP'\in \mathscr{X}_{R,\Gamma}$ although $P\notin \mathscr{X}_{R,\Gamma}$ and we are done.
\end{proof}

\begin{Remark} 
\label{remark.trivial.grading}
The same result as in Proposition \ref{main.prop.matrix.algebra} can be obtained in an analogous way for an elementary grading on $M_{3}(R)$ given by a triple $(g_{1}, g_{2}, g_{3})$ such that exatly two among $g_1$, $g_2$, $g_3$ are equal. Let us observe here that $M_{n}(K)$, $n\geq 2$) with the trivial grading satisfies the primeness property for graded central polynomials, as shown in \cite{regev1979primeness}.
\end{Remark}

Now, since $K$ is algebraically closed, and $G$ is a finite abelian group, we have exactly $|G|$ irreducible characters of $G$. Therefore, by combining Theorems \ref{Diniz.Claudemir.Theorem} and ~\ref{elementary}, Proposition~\ref{main.prop.matrix.algebra}, and Remark~\ref{remark.trivial.grading}, we obtain the following result.

\begin{Theorem}
\label{M2.M3.Trivial}
{Let $K$ be an algebraically closed field,} then for each $k\in \{2,3\}$, there is no non-trivial grading $\Gamma$ on $M_{k}(K)$, given by a finite abelian group, such that $(M_{k}(K),\Gamma)$ satisfies the primeness property for graded central polynomials.
\end{Theorem}

Theorems \ref{M2.M3.Trivial} and \ref{Diniz.Claudemir.Theorem} give us ground to suspect that over algebraically closed fields of characteristic $0$, there is no grading by an abelian group on $M_{n}(K)$, for $n>3$, such that $M_{n}(K)$ has the primeness property for graded central polynomials. We state it in the following conjecture:

\begin{Conjecture}
Let $G$ be a finite nontrivial abelian group. Over an algebraically closed field $K$ of characteristic zero, no nontrivial $G$-grading $\Gamma$ on $M_n(K)$, with $n>3$, satisfies the primeness property for graded central polynomials.
\end{Conjecture}

We observe that the assumptions of $G$ being abelian and $K$ being an algebraically closed field cannot be removed. For, consider the crossed product gradings on $M_{60}(K)$ (for $K = \mathbb{R}$ or $\mathbb{C}$) and on $M_5(\mathbb{R})$ by the alternating group $A_5$ and by the cyclic group $\mathbb{Z}_5$, respectively. By applying Theorem~\ref{Diniz.Claudemir.Theorem}, one can verify that these graded algebras satisfy the primeness property for graded central polynomials.

\section{\texorpdfstring{Primeness property for central polynomials in $\mathbb{Z}_{2}$-regular gradings}{Primeness property for central polynomials in Z2-regular gradings}}

In this section, we will work with the primeness property in the ordinary sense. If $R$ is a PI-algebra, we say that $R$ has the primeness property for central polynomials if, for every  $f$, $g \in K\langle X\rangle$ in disjoint variables such that $fg$ is a central polynomial, then both $f$ and $g$ are central polynomials. 

Our main goal is to show that the $\mathbb{Z}_{2}$-graded regular algebras with minimal regular decomposition satisfy the primeness property for central polynomials. To this end we apply techniques developed in \cite{L.P.K.3}.  More precisely, we use the following idea. If $A$ is a $\mathbb{Z}_{2}$-graded regular algebra with minimal regular decomposition, it satisfies the primeness property for central polynomials because there exists an embedding of $E$ into $A$ (see \cite[Theorem 17]{L.P.K.3}), and $E$ itself satisfies this property by \cite[Proposition 3.5]{diniz2015primeness}. We emphasize that this is not true in general: it may happen that an algebra $R$ does not satisfy the primeness property, while a subalgebra of $A$ does. For instance, if $R = E_{3} = \langle v_{1}, v_{2}, v_{3} \rangle$ is the Grassmann algebra in three generators (that is $E(3)$ is the Grassmann algebra of a vector space of dimension 3), then $R$ does not satisfy the primeness property, since $x[y,z]$ is a proper central polynomial but $x$ is not central in $R$. However, if we consider the subalgebra $B = \operatorname{Span}\{1, v_{1}v_{2}\}$, then $B$ trivially satisfies the primeness property for central polynomials, because it is commutative. In the direction of \cite{L.P.K.3}, this further reinforces the idea that the Grassmann algebras “almost” completely determine a $\mathbb{Z}_{2}$-graded regular algebra whose regular decomposition is minimal.

Let $A$ be a $\mathbb{Z}_{2}$-graded regular algebra such that its regular decomposition is minimal. Then we know that $M^{A} = M^{E}$, and thus its bicharacter $\tau$ is given by $\tau(0,0)=\tau(0,1)=\tau(1,0)=1$ and $\tau(1,1)=-1$. Moreover, we know that there exists a graded embedding of $E$ into $A$, and without loss of generality, we may assume that $E \subseteq A$. 

\begin{Definition}  An element $a \in A_{1}$ satisfying $ab = 0$ for all $b \in A_{1}$ will be called a \textsl{$(1,1)$-annihilator} of $A$.
\end{Definition}
 The set of all $(1,1)$-annihilators of $A$ will be denoted by $\mathfrak{T}$.

\begin{Lemma} The set $\mathfrak{T}$ is a graded ideal of $A$ such that $\mathfrak{T}_0=0$, and $\mathfrak{T}^2=0$. The quotient $A/\mathfrak{T}$ is a non-zero $\mathbb{Z}_{2}$-graded regular algebra with bicharacter $\tau$. 
\end{Lemma}

\begin{proof} The first two statements are obvious. Since $E\subseteq A$, it is immediate that $A/\mathfrak{T}$ is a non-zero $\mathbb{Z}_{2}$-graded regular algebra with bicharacter $\tau$.    
\end{proof}

As the regular decomposition of $A$ is minimal, and $\tau(1,1)=-1$, we have, by \cite[Lemma 8]{L.P.K.3} that the $\mathbb{Z}_2$-graded identities of $A$ and $E$ are the same, hence $T(A)=T(E)$. Moreover, since $E\subseteq A$, with the same argument as in \cite[Lemma 3.4]{diniz2015primeness}, we obtain the following technical result.

\begin{Lemma}
\label{Aux.Lemma.ordinary.case}
Let $\Gamma\colon A=A_{0}\oplus A_{1}$ be a $\mathbb{Z}_{2}$-graded regular algebra whose regular decomposition is minimal. If $f(x_{1},\ldots, x_{m})$ is a multilinear polynomial that is not an identity for $A$, then there exist $a_{1}$, \dots, $a_{m}$ in $E\subseteq A$ such that $0\neq f(a_{1},\ldots, a_{m})\in Z(E)\subseteq Z(A)$.
\end{Lemma}

\begin{Lemma} 
\label{Aux.equivalence}
Let $\Gamma\colon A = A_0 \oplus A_1$ be a $\mathbb{Z}_2$-graded regular algebra whose regular decomposition is minimal. If $B:=A/\mathfrak{T}$ satisfies the primeness property for central polynomials, then so does $A$.
\end{Lemma}
\begin{proof} Take $f(x_{1},\ldots, x_{r})$ and $g(x_{r+1},\ldots, x_{s})$ two multilinear polynomials in distinct variables such that 
\[
h(x_{1},\ldots,x_{r}; x_{r+1},\ldots, x_{s}):=f(x_{1},\ldots, x_{r})g(x_{r+1},\ldots, x_{s})
\]
is a central proper polynomial in $A$. Since $h\notin T(A)$, then there exists $a_{1}$, \dots, $a_{s}\in A$ such that $h(a_{1},\ldots, a_{s})\neq 0$. If $f\in T(A)$ then it implies that $h(a_{1},\ldots, a_{s})=0$ which is a contradiction. Similarly $g\notin T(A)$. Therefore, we need to show that $f$ and $g$ are central polynomials in $A$. 

Suppose there exists $0\neq a+\mathfrak{T}\in Z(B)\cap B_{1}$ where $B=A/\mathfrak{T}$. Take an element $c_{1}\in A_{1}$, then if $ac_{1}\notin \mathfrak{T}$, we have
\[
 (a+\mathfrak{T})(c_{1}+\mathfrak{T})=(c_{1}+\mathfrak{T})(a+\mathfrak{T})=\tau(1,1)(a+\mathfrak{T})(c_{1}+\mathfrak{T}).
\]
This implies $\tau(1,1)=1$ which is a contradiction. Therefore $ac_{1}\in \mathfrak{T}$, for every $c_{1}\in A_{1}$, thus $ac_{1}=0$. We conclude $Z(B)=A_{0}+\mathfrak{T}$.

By Lemma \ref{Aux.Lemma.ordinary.case}, there exist $a_{1}$, \dots, $a_{s}\in E$ such that $0\neq h(a_{1},\ldots, a_{s})\in Z(E)\subseteq Z(B)$. Therefore there exists $0\neq b\in E_{1}$ such that $h(a_{1},\ldots, a_{s})b\neq 0$, and thus $h(a_{1},\ldots, a_{s})\notin \mathfrak{T}$. Hence $h$ is not an identity of $B$, and $h$ is a central polynomial of $B$. By hypothesis, $f$ and $g$ are central polynomials in $B$. Take arbitrary elements $a_{1}$, \dots, $a_{r}\in A$, then 
\[
f(a_{1}+\mathfrak{T},\ldots, a_{r}+\mathfrak{T})\in Z(B).
\]
 Thus $f(a_{1},\ldots, a_{r})$ is of the form $f(a_{1},\ldots, a_{r})=s_{0}+t$, $s_{0}\in A_{0}$ and $t\in \mathfrak{T}$. If $t\neq 0$, then $f(a_{1},\ldots, a_{r})\in Z(A)$. Suppose $t\neq 0$, and let $c=c_{0}+c_{1}$ be an element in $A$. Then
\begin{align*}
    f(a_{1},\ldots, a_{r})c &=(s_{0}+t)(c_{0}+c_{1})=s_{0}c_{0}+tc_{0}+s_{0}c_{1}\\
    &= c_{0}s_{0}+c_{0}t+c_{1}s_{0}=cf(a_{1},\ldots, a_{r})
\end{align*}
that is, $f(a_{1},\ldots, a_{r})\in Z(A)$.  Analogously one proves $g(a_{1},\ldots, a_{r})\in Z(A)$.
\end{proof}

In the next lemma we proceed in the same way as \cite[Proposition 3.5]{diniz2015primeness}.

\begin{Lemma} 
\label{Primeness.quotiente}
Let $\Gamma\colon A=A_{0}\oplus A_{1}$ be a $\mathbb{Z}_{2}$-graded regular algebra whose regular decomposition is minimal. Then, $B:=A/\mathfrak{T}$ satisfies the primeness property for central polynomials.
\end{Lemma}
\begin{proof} Let $f(x_{1},\ldots, x_{r})$ and $g(x_{r+1},\ldots, x_{s})$ be multilinear polynomials in distinct variables such that 
\[
h(x_{1},\ldots,x_{r}; x_{r+1},\ldots, x_{s}):=f(x_{1},\ldots, x_{r})g(x_{r+1},\ldots, x_{s})
\]
is a central proper polynomial for $B$. If $f\in T(B)$, then $f(a_{1},\ldots, a_{r})\in \mathfrak{T}$, for all $a_{1}$, \dots, $a_{r}\in A$, but this contradicts Lemma \ref{Aux.Lemma.ordinary.case}. Similarly $g\notin T(B)$. Moreover, $f$ and $g$ must be of degree $\ge 2$, because if for instance $\deg f = 1$, then for each $b \in Z(B)$, there exists a homogeneous element $d \in A$ such that $f(d)b$ has degree $1$. Now, write
\[
g(x_{r+1},\ldots, x_{s})=\sum_{\sigma\in S\{r+1,\ldots, s\})} \gamma_{\sigma} x_{\sigma(1)}\cdots x_{\sigma(r)},\quad \gamma_{\sigma}\in K, \text{ for all } \sigma\in S_{r}
\]
where $S\{r+1,\ldots, s\}$ is the group of all permutations of the elements $r+1$, \dots, $s$.

Suppose there exist homogeneous elements $a_{r+1}+\mathfrak{T}$,\dots, $a_{s}+\mathfrak{T}\in B$ such that $0\neq g(a_{r+1}+\mathfrak{T},\ldots, a_{s}+\mathfrak{T})\in B_{1}$. By regularity, there exists $0\neq \gamma\in K$ such that
\[
g(a_{r+1}+\mathfrak{T},\ldots, a_{s}+\mathfrak{T})=\gamma a_{r+1}\cdots a_{s}+\mathfrak{T}.
\]
Since $(a_{r+1}\cdots a_{s})+\mathfrak{T}\neq 0$, there exists a homogeneous element $0\neq t\in A_{1}$ such that $a_{r+1}\cdots a_{s}t\neq 0$.  Thus there exists $k>s$ such that the elements $a_{r+1}$,\dots, $a_{s}$ are in a finite-dimensional Grassmann algebra $F_{k}$. Without loss of generality we assume that $F_{k}$ is contained in $E$ and $F_{k}$ is generated by 1 and $e_{1}$, \dots, $e_{k}$. Consider the subalgebra $L$ generated by $1$ and $e_{j}$, with $j>k$. By construction, the product of any nonzero element in $L$ with $g(a_{r+1},\ldots, a_{s})$ is nonzero.

Since $L\cong E$, by Lemma \ref{Aux.Lemma.ordinary.case} there exist $l_{1}$, \dots, $l_{r}\in L$ with $0\neq f(l_{1},\ldots, l_{r})\in Z(L)$, that is, there exists $\kappa\in K\setminus\{0\}$ with $f(l_{1},\ldots, l_{r})= \kappa l_{1}\cdots l_{r}\in Z(L)\setminus\{0\}$.

By definition of $L$, there exists  $w\in L_{1}\subseteq A_{1}$ such that $l_{1}\cdots l_{r}a_{r+1}\cdots a_{s}w\neq 0$. Therefore, we get
\[
0\neq f(l_{1}+\mathfrak{T},\ldots, l_{r}+\mathfrak{T})g(a_{r+1}+\mathfrak{T},\ldots, a_{s}+\mathfrak{T})= \kappa\gamma l_{1}\cdots l_{r}a_{r+1}\cdots a_{s}+\mathfrak{T}\in Z(B)\cap B_{1}
\]
which is a contradiction. Hence $g$ is a central polynomial. In a similar way we see that $f$ is a central polynomial, too. 
\end{proof}
Combining Lemmas \ref{Aux.equivalence} and \ref{Primeness.quotiente}, we obtain the following theorem.
\begin{Theorem}
\label{main.theorem}
Let $\Gamma\colon A=A_{0}\oplus A_{1}$ be a $\mathbb{Z}_{2}$-graded regular algebra whose regular decomposition is minimal. Then, $A$ satisfies the primeness property for central polynomials. 
\end{Theorem}

\begin{Corollary} Let $\Gamma\colon A=A_{0}\oplus A_{1}$ be a $\mathbb{Z}_{2}$-graded regular algebra. Then, $A$ satisfies the primeness property for central polynomials. 
\end{Corollary}
\begin{proof}  If $A$ is a $\mathbb{Z}_{2}$-graded regular algebra whose regular decomposition is non-minimal, then $A$ is a commutative algebra, and therefore, $A$ trivially satisfies the primeness property for central polynomials. If the regular decomposition of $A$ is minimal, then the result follows from Theorem \ref{main.theorem}.
\end{proof}

An interesting phenomenon occurs with a regular $\mathbb{Z}_{2}$-graded algebra with minimal regular decomposition. More precisely, none of its finitely generated non-commutative graded subalgebras satisfies the primeness property for central polynomials as we show in the following proposition. 

\begin{Proposition} Let $A$ be a $\mathbb{Z}_{2}$-graded regular algebra with a minimal regular decomposition, and let $B$ be a finitely generated graded subalgebra of $A$. If $B$ satisfies the primeness property, then $B$ is commutative.  
\end{Proposition}
\begin{proof} According to \cite[Theorem 27]{L.P.K.3} it is enough to prove that if $B=\langle a_{1},\ldots, a_{n}\rangle$, where $a_{1}$, \dots, $a_{n}\in A_{1}$, then $B$ does not satisfy the primeness property for central polynomials. Since $a_{1}\cdots a_{n}a_{i}=0$, for all $1\leq i\leq n$, we can write $B=K\oplus B(n)$, where $B(n)$ is the following subalgebra of $B$
\[
B(n):=\operatorname{Span}_{K}\{a_{j_{1}}\cdots a_{j_{k}}\mid j_{1},\ldots, j_{k}\in \{1,\ldots, n\},\quad k\leq n\}.
\] An easy computation shows $B_1=B(n)_1$.
 Consider now
\[
 l:=\max\Big\{k\mid a_{j_{1}}\cdots a_{j_{k}}\neq 0,\quad j_{1},\ldots, j_{k}\in \{1,\ldots, n\}\Big\}
\]
then, $l\leq n$ and $x_{1}\cdots x_{l+1}\in T(B(n))$. In particular, if $a_{j_{1}}\cdots a_{j_{l}}\neq 0$, then $a_{j_{1}}\cdots a_{j_{l}}\in Z(B)$. Without loss of generality, suppose $a_{1}\cdots a_{l}\neq 0$. 
Take $k\in \mathbb{N}$ as the largest integer such that $2k\leq l$, and define
\[
t(x_{1},\ldots, x_{2k})=[x_{1},x_{2}]\cdots [x_{2k-1},x_{2k}].
\]
Notice that for all $a$, $b\in B$, with $a=\delta+a'$ and $b=\gamma+b'$, where $\delta$, $\gamma\in K$ and $a'$$ $$,b'\in B(n)$, we have 
$[a,b]=[a',b']$. Thus, we get
\[
t(b_{1},\ldots, b_{2k})\in B(n),\quad \text{for all}\quad b_{1}.\ldots,b_{2k}\in B,
\]
If some $b_{i}\in B_{0}$, we automatically have $t(b_{1},\ldots, b_{2k})=0$. Moreover, since $Z(B)$ contains the subalgebra generated by $1$ and monomials $a_{q_{1}}\cdots a_{q_{u}}$, with $u\leq l$ and $q_{u}$ even, we conclude that
\[
t(b_{1},\ldots, b_{2k})\in Z(B).
\]
Now, consider the polynomial $f(x,x_{1},\ldots, x_{2k}):=x t(x_{1},\ldots,x_{2k})$, where $x$ is a variable different from $x_{i}$, $1\leq i\leq 2k$. If for some homogeneous elements $a$, $b_1$, \dots, $b_k$ we have $f(a,b_{1},\ldots,b_{2k})\neq 0$, then $b_{1}$, \dots, $b_{2k}\in B_{1}=(B(n))_{1}$ and
\[
f(a,b_{1},\ldots, b_{2k}) = a[b_{1},b_{2}]\cdots [b_{2k-1},b_{2k}] = 2^{k} a b_{1}\cdots b_{2k}.
\]
Since $x_{1}\cdots x_{l+1}\in T(B(n))$, it follows that either $a\in K$, or $a=a_{i}$ for some $1\leq i\leq n$. In both cases, $f(a,b_{1},\ldots,b_{2k})\in Z(B)$.

Moreover, $f\notin T(B)$ because if $l$ is odd, then $l=2k+1$ and 
\[
f(1,a_{1},\ldots, a_{l-1}) = 2^{k} a_{1}\cdots a_{l-1} \neq 0.
\]
If $l$ is even, then $l=2k$ and 
\[
f(1,a_{1},\ldots, a_{l}) = 2^{k} a_{1}\cdots a_{l} \neq 0.
\]
Consequently, $f$ is a proper central polynomial for $B$, whereas $x$ is not. We conclude that $B$ does not satisfy the primeness property for central polynomials.
\end{proof}

\bibliographystyle{abbrv}
\bibliography{Ob.reg.bib}

\end{document}